\documentclass[12pt,a4paper,reqno]{amsart}

\usepackage{algpseudocode}

\usepackage{amsthm}
\usepackage{tikz}
\usepackage{fullpage}
\usepackage{float}
\usepackage{multirow}
\usepackage{amsmath}
\usepackage{amssymb}
\usepackage{amsthm}
\usepackage{amscd}
\usepackage{amsfonts}
\usepackage{array}
\usepackage{hyperref}
\usepackage{url}
\usepackage{mleftright} 
\usepackage{float}
\usepackage{multirow}
\usepackage{hyperref}
\usepackage{ulem}
\usepackage{soul}
\PassOptionsToPackage{lowtilde}{url}
\usepackage{xcolor}
\usepackage{cancel}

\newtheorem{theorem}{Theorem}[section]
\newtheorem{lemma}[theorem]{Lemma}
\newtheorem{corollary}[theorem]{Corollary}
\newtheorem{proposition}[theorem]{Proposition}
\newtheorem{conjecture}[theorem]{Conjecture}
\newtheorem{example}[theorem]{Example}

\def\1{{\bf 1}}

\def\vu{\mbox{\boldmath $u$}}
\def\vv{\mbox{\boldmath $v$}}

\DeclareMathOperator{\diam}{diam}
\DeclareMathOperator{\rad}{rad}

\DeclareMathOperator{\spec}{sp}

\def\BB{\mbox{\boldmath $B$}}
\def\Q{\mbox{\boldmath $Q$}}
\def\I{\mbox{\boldmath $I$}}

\def\Z{\mathbb Z}
\newcommand{\cal}{\mathcal}

\def\purple{\textcolor{purple}}

\definecolor{greeen}{RGB}{47, 141, 26}

\usepackage{notation}

\date{}
\begin{document}
\title{On supertoken graphs}

\author{M\'onica A. Reyes}
\address{Departament de Matem\`atica, Universitat de Lleida, Igualada (Barcelona), Catalonia}
\email{monicaandrea.reyes@udl.cat}

\author{Cristina Dalf\'o}
\address{Departament de Matem\`atica, Universitat de Lleida, Igualada (Barcelona), Catalonia}
\email{cristina.dalfo@udl.cat}

\author{Miquel \`Angel Fiol}
\address{Departament de Matem\`atiques, Universitat Politècnica de Catalunya, Barcelona, Catalonia; Barcelona Graduate School of Mathematics, Barcelona, Catalonia; Institut de Matem\`atiques de la UPC-BarcelonaTech (IMTech), Barcelona, Catalonia}
\email{miguel.angel.fiol@upc.edu}

\begin{abstract}
We generalize the concept of token graphs to obtain supertoken graphs. In the latter case, there can be more than one token in a vertex. We formally define supertoken graphs and establish their basic properties.
Moreover, we provide some bounds and exact values on the independence number, clique number, and chromatic number of these graphs.  
Finally, we construct a new infinite family of graphs, which we call the $p$-augmented 2-token graphs of cycles, and study their properties, including the spectral radius or largest adjacency eigenvalue.
\end{abstract}

\maketitle



\section{Introduction}
\label{sec:supertoken}

Let us consider the following communication model. Given a set of $n$ possible symbols, corresponding to the vertices of a graph $G$, we assume that instead of single symbols or (ordered) strings of them, we transmit multisets of $k$ (not necessarily different) symbols. Moreover, we suppose that the probability of error in each symbol is small enough to assume that, in the transmission of a multiset, at most one symbol is confused. In other words, the probability of two or more wrong symbols can be neglected.
In this framework, the so-called confusability graph is a $k$-supertoken graph, whose definition follows.\\
Given a graph $G$ on $n$ vertices and an integer $k\ge 1$, 
 the \textit{$k$-supertoken graph}$F_1^{k\times 1}(G)$  is defined as follows: Every vertex of ${\cal F}_k(G)$ corresponds to a way of placing $k$ (indistinguishable) tokens in some of the $n$ (not necessarily distinct) vertices of $G$. 
Thus, we represent each vertex $\vu$ of ${\cal F}_k(G)$ by a (non-negative) $n$-vector or $n$-sequence
$$
\vu=(u_1,u_2,\ldots,u_n)\equiv u_1u_2\ldots u_n\quad \mbox{with } \sum_{i=1}^{n}u_i=k.
$$ 
 Moreover, two of its vertices, $\vu$ and $\vv$, are adjacent if one token of the multiset representing $\vu$ is moved along an edge of $G$ to another vertex, so obtaining the multiset representing $\vv$. Then, notice that (the multisets of) $\vu$ and $\vv$ have $k-1$ elements in common.
This kind of token graph was introduced by Hammack and Smith~\cite{HaSm17}, who named them \textit{reduced $k$-th power of graphs}, and provided
a construction of minimum cycle bases for them.

Supertoken graphs are a particular case of a broader family that we call 
\textit{generalized token graphs}.
In such graphs, each vertex represents a configuration of $k$ (undistinguished or distinguished) tokens placed on some (different or equal) vertices of $G$. Moreover, two vertices (or configurations) are adjacent when one configuration can be obtained from the other by moving in $G$ (one or more) tokens along the incident edges. {Depending on the nature and position of the tokens and the adjacency rules listed in Table \ref{tab:tokens}, we have the following different families of generalized token graphs.}

\begin{itemize}
\item
$F_k(G)$: The \textit{$k$-token graph} of $G$ (Fabila-Monroy, Flores-Pe\~naloza, Huemer, Hurtado, Urrutia, and Wood \cite{ffhhuw12}) or \textit{$k$-th symmetric power} of $G$ (Audenaert, Godsil, Royle, and Rudolph   \cite{agrr07});
\item
${\cal F}_k(G)$: The \textit{$k$-supertoken graph} of $G$ (Baskoro, Dalf\'o, Fiol, and Simanjuntak \cite{bdfs24}) or \textit{$k$-th reduced power graph} of $G$ (Hammack and Smith \cite{HaSm17});
\item
$G^{\Box k}$: \textit{$k$-th direct product graph} of $G$ (that is, the direct product of $G$ by itself $k$ times);
\item
$G^{\boxtimes k}$: \textit{$k$-th strong product graph} of $G$ (or, in information theory, the \textit{$k$-th confusability graph} of $G$). 
\end{itemize}

\begin{figure}[t]
	\begin{center}
		\includegraphics[width=6cm]{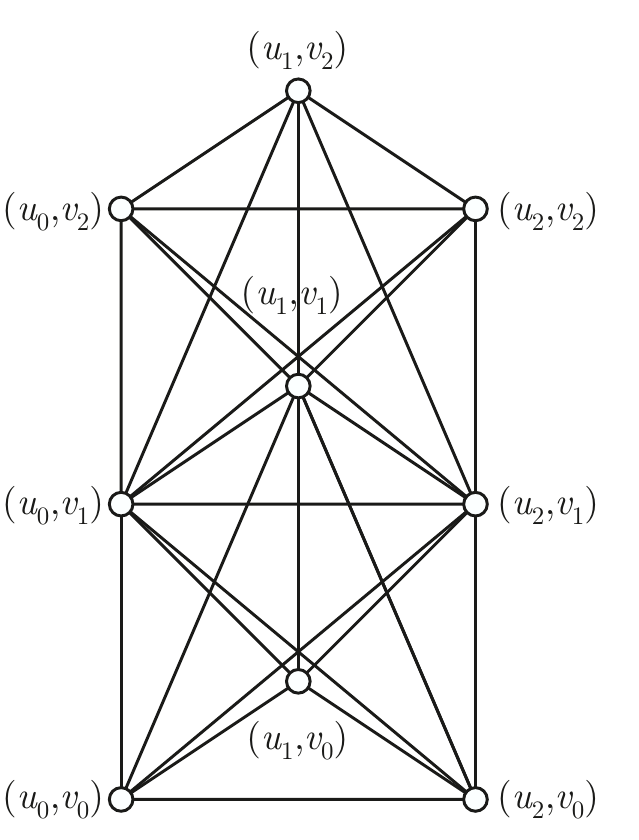}
	\end{center}
	\vskip-.25cm
	\caption{The strong product $C_3\boxtimes P_3$.}
	\label{fig:C_3xP3}
\end{figure}

\begin{table}[t]
\begin{center}
\begin{tabular}{|c|c|c|c|}
 \hline
Graph operation & Distinguished tokens & 
Max. no. of tokens & Max. no. of tokens\\
 & & on each vertex &  moved at each step\\
 \hline
$F_k(G)$ & No & 1 & 1 \\
\hline
${\cal F}_k(G)$ & No & $k$ & 1\\
 \hline
 $G^{\Box k}$ & Yes & $k$ & 1 \\
  \hline
  $G^{\boxtimes k}$ & Yes & $k$ & $k$\\
 \hline
\end{tabular}
\end{center}
\caption{Different operations of a graph with itself defining a family of generalized token graphs.}
\label{tab:tokens}
\end{table}

Recall that, given the graphs $G_1, \ldots,G_k$, their \textit{strong product}, denoted by $G_1\boxtimes \cdots \boxtimes G_k$, is the graph with vertex set the Cartesian product $V(G_1) \times \cdots \times V(G_k)$, where vertices $(v_1, \ldots, v_k)$ and $(u_1, \ldots, u_k)$ are adjacent if and only if for every $i \in [1,k]=\{1,\ldots,k\}$,
either $v_i = u_i$ or $v_iu_i \in E(G_i)$. As shown in Table \ref{tab:tokens}, we denote the $k$ strong products of $G$ by itself as $G^{\boxtimes k}$. Figure \ref{fig:C_3xP3} illustrates the strong product of the cycle $C_3$ with the path $P_3$. 

Lov\'asz \cite{l79} showed that $\sqrt[k]{\alpha(C_n^{\boxtimes k})}\le \sqrt[2]{\alpha(C_5^{\boxtimes 2})}$ for every $k\ge 2$, {where $\alpha$ is the independence number}. This makes it interesting to study the behavior of $\alpha(C_n^{\boxtimes 2})$ for a cycle $C_n$ with $n>5$. With this aim in mind, let us consider other interpretations. First, it is known that $a(n):=\alpha(C_n^{\boxtimes 2})=\lfloor\frac{n}{2}\lfloor\frac{n}{2}\rfloor\rfloor$. 
The sequence 
$$
\{a(n)\}_{n\ge 1}=\{1,1,1,4,5,9,10,16, 18,25,27,36,39,49,\ldots\}
$$ 
corresponds to $A189889$ in OEIS \cite{Sl}, as the maximum non-attacking kings in an $n\times n$ toroidal chessboard; see Watkins \cite[Thm. 11.1]{w07}. Thus, an alternative definition of $a(n)$ is the independence number of the Cayley graph on the Abelian group $\Z_n\times \Z_n$ with generators $(\pm e_1, \pm e_2)\neq (0,0)$, where $e_i\in \{0,1\}$ for $i=1,2$. This suggests representing the strong product $G=C_n\boxtimes C_n$ as a graph associated with plane tessellations as follows (see Yebra, Fiol, Morillo, and Alegre  \cite{yfma85}). Every vertex of $G$ is represented by a unit square with center a lattice point $(i,j)$, for $i,j\in \Z_n$, and $(i,j)$ is adjacent to the eight vertices $(i\pm 1,j)$, $(i,j\pm 1)$, $(\pm i,\pm j)$ and $(\pm i,\mp j)$. Then, two vertices are adjacent if one of their coordinates differs at least by 1 unit. Using this approach, Figure \ref{fig:C7XC7} (left) shows a maximum independent set of 10 vertices in $C_7\boxtimes C_7$.
Curiously enough, as it was shown by de Alba, Carballosa,  Lea\~nos, and Rivera \cite{aclr20}, $a(n)$ is also the independence number of the 2-token graph $F_2(C_n)$ of the cycle $C_n$. In Figure \ref{fig:C7XC7} (middle), we show $F_2(C_7)$ with its 10 independent vertices (the white vertices). It would be interesting to find a correspondence between the independent vertices of $C_n\boxtimes C_n$ and $F_2(C_n)$. By considering this last interpretation, we can also say that $a(n)$ is the maximum number of edges of an $n$-cycle graph $C_n^+$ with chords not containing any triangle with the edges of the cycle, see Figure \ref{fig:C7XC7} (middle and right), where the equivalence between the independent vertices of $F_2(C_7)$ and the edges of $C_7^+$ is apparent.

\begin{figure}[t]
	\begin{center}
 \includegraphics[width=15cm]{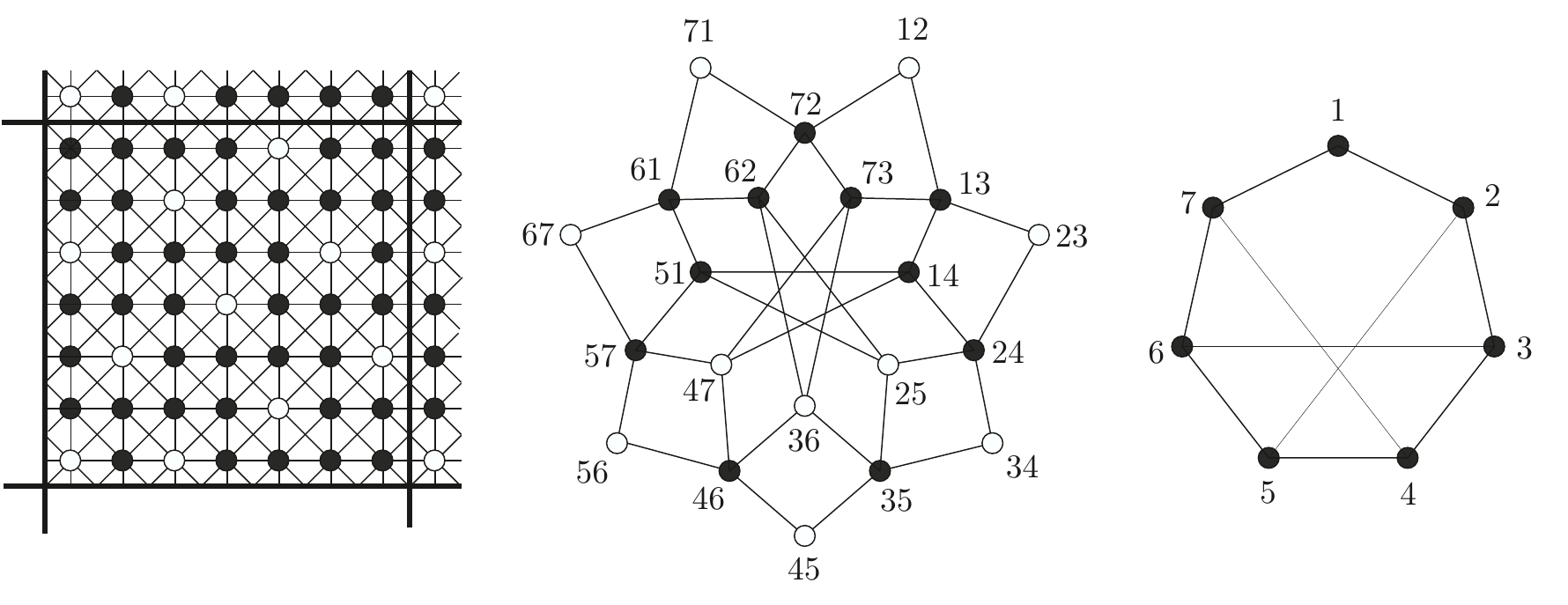}
	\end{center}
	\vskip-.25cm
	\caption{Left: A maximum independent set in  $C_7\boxtimes C_7$ (white vertices). Middle: A maximum independent set in $F_2(C_7)$ (white vertices). Right: The cycle $C_7$ with three chords and no triangles.}
	\label{fig:C7XC7}
\end{figure}

This paper is structured as follows. In the following section, we provide the basic properties of supertoken graphs. In Section \ref{sec:ind-number}, we give a lower bound on the independent number of supertoken graphs. In Sections \ref{sec:clique} and \ref{sec:chrom}, the clique number {and chromatic number} of these graphs are established, respectively. Finally, in Section \ref{sec:aug-token-cicles}, we construct a new infinite family of graphs that we call the $p$-augmented 2-token graphs of cycles, and we study their properties, among them, the spectral radius, that is, their largest adjacency eigenvalue.


\section{Fundamental properties of supertoken graphs}

Let $G$ be a graph with $n$ vertices. The number of vertices of ${\cal F}_k(G)$ is the number of combinations with repetitions $CR_k^n$ of $n$ elements taken $k$ at a time, that is,
$$
|V({\cal F}_k(G))|=|CR_k^n|= \binom{n+k-1}{k} = \binom{n+k-1}{n-1}.
$$

To determine the number of edges in a supertoken graph, note that each edge of $G$ gives us $|CR_{k-1}^{n}| = {\binom{n+k-2}{k-1}}$ edges in the supertoken ${\cal F}_k(G)$. Consequently, if $G$ has $m$ edges, the total number of edges in ${\cal F}_k(G)$ is given by 
$$| E ({\cal F}_k(G)|=m |CR_{k-1}^{n}| = m{\binom{n+k-2}{k-1}}.$$
For more details and other generalizations of token graphs, see Song, Dalf\'o, Fiol, Mora, and Zhang \cite{sdfmz26}.
{For example, in the case of the cycle $C_n$ and the complete graph $K_n$ on $n$ vertices, the numbers of edges of their 2-supertoken graphs are $|E({\cal F}_2(C_n))|=n^2$ and $|E({\cal F}_2(K_n))|=\frac{1}{2}n^2(n-1)$. Figure \ref{k4_token_supertoken} presents the examples of the 2-supertoken and 3-supertoken graphs of $K_4$,
{where the edges of each color in $K_4$ give rise to the edges of the same color in the supertokens. Moreover,}  in the case of ${\cal F}_3(K_4)$, a regular partition of the vertex set can be constructed. This partition consists of three classes: one formed by the outer vertices with degree 3, another by all vertices with degree 6, and the last by the central vertices with degree 9. Using this regular partition, part of the Laplacian spectrum of 
${\cal F}_3(K_4)$ can be determined from the spectrum of the quotient Laplacian matrix {$\Q$} associated with the partition. {Namely,
$$
\Q =
\left(
\begin{array}{ccc}
3 & -3 & 0\\
-1 & 3 & -2\\
0 & -6 & 6
\end{array}
\right),
$$
with eigenvalues $0,6\pm\sqrt{6}$.}

\begin{figure}[h]
  \centering
\includegraphics[width=0.9\textwidth]{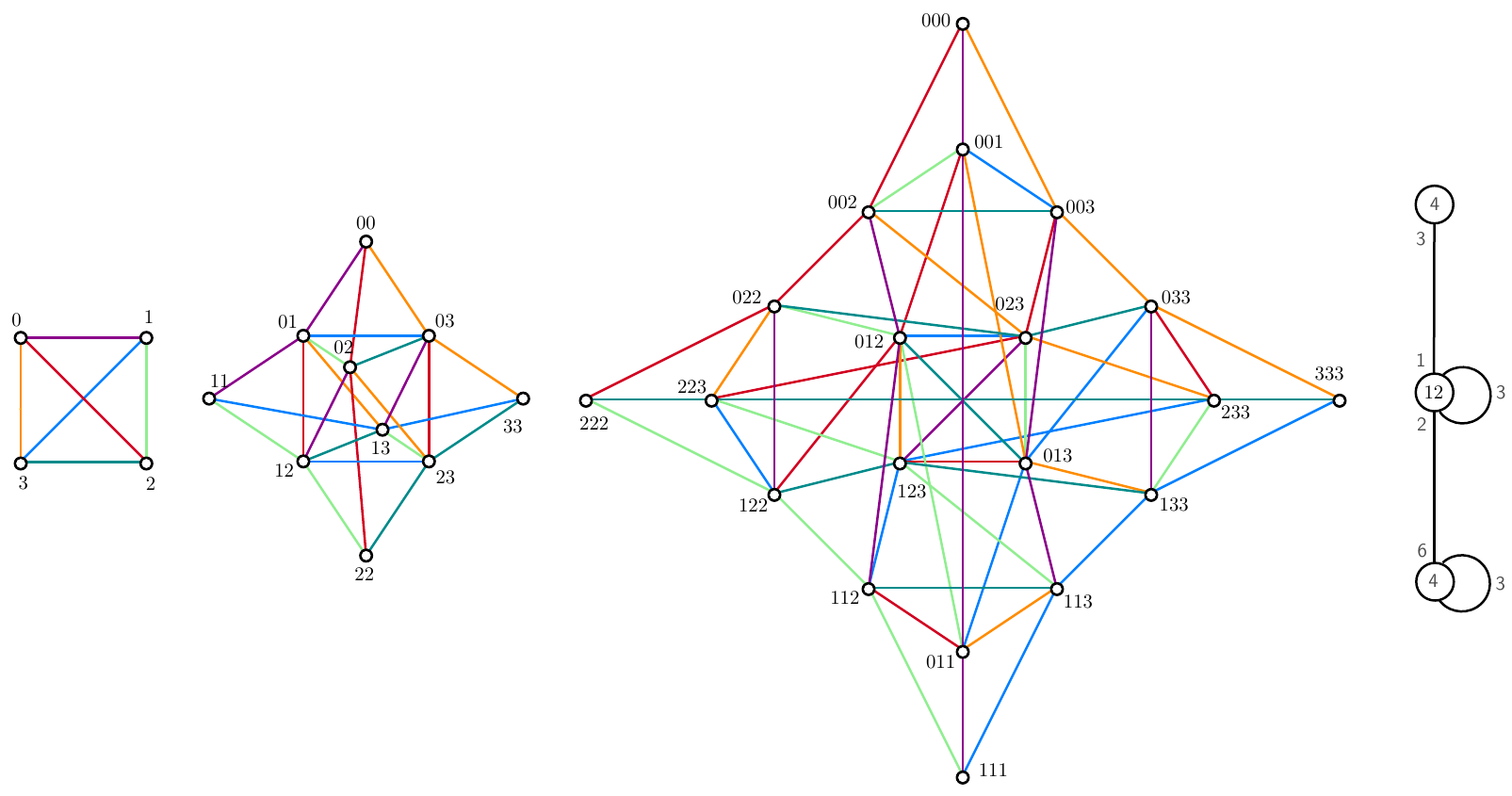}
  \caption{{The complete graph $K_4$, its $2$-supertoken, its $3$-supertoken, and a regular partition of its $3$-supertoken.}}
  \label{k4_token_supertoken}
\end{figure}

Moreover, an interesting property of the supertoken of the path graph $P_n$ is discussed in Baskoro, Dalf\'o, Fiol, and Simanjuntak \cite{bdfs24}, where it is noted that the 2-supertoken graph ${\cal F}_2(P_n)$ is isomorphic to the 2-token graph $F_2(P_{n+1})$. \\  In the same article, the authors also established several results concerning supertoken graphs ${\cal F}_k(G)$. Specifically, for a graph $G=(V,E)$ with $n$ vertices, diameter $d$, and radius $r$, they proved the following properties:
\begin{enumerate}
\renewcommand{\labelenumi}{(\roman{enumi})}
\item 
The diameter of 
${\cal F}_k(G)$ is $\diam({\cal F}_k(G)) =kd$.
\item 
The radius of ${\cal F}_k(G)$ satisfies 
$\rad({\cal F}_k(G))\leq kr$.
\item 
The metric dimension of 
${\cal F}_k(G)$ satisfies 
$\dim({\cal F}_k(G)) \leq |V|$.
\end{enumerate}
An interesting distinction arises when comparing token graphs and supertoken graphs.}
{Thus, whereas the token graph $F_k(G)$ is not, in general, a subgraph of $F_{k+1}(G)$, the supertoken graph ${\cal F}_k(G)$ is always a subgraph of ${\cal F}_{k+1}(G)$.} Also, $F_k(G)$ is always a subgraph of ${\cal F}_{k}(G)$. 
Then, we have the following relation between their spectra.
\begin{enumerate}
\renewcommand{\labelenumi}{(\roman{enumi})}
\item 
The eigenvalues of $F_{k-1}(G)$ interlace the eigenvalues of ${\cal F}_{k-1}(G)$.
\item 
The eigenvalues of ${\cal F}_{k-1}(G)$ interlace the eigenvalues of ${\cal F}_{k}(G)$.
\end{enumerate}

\section{Independence number of supertoken graphs}
\label{sec:ind-number}

\begin{figure}[t]
	\begin{center}
	\includegraphics[width=14cm]{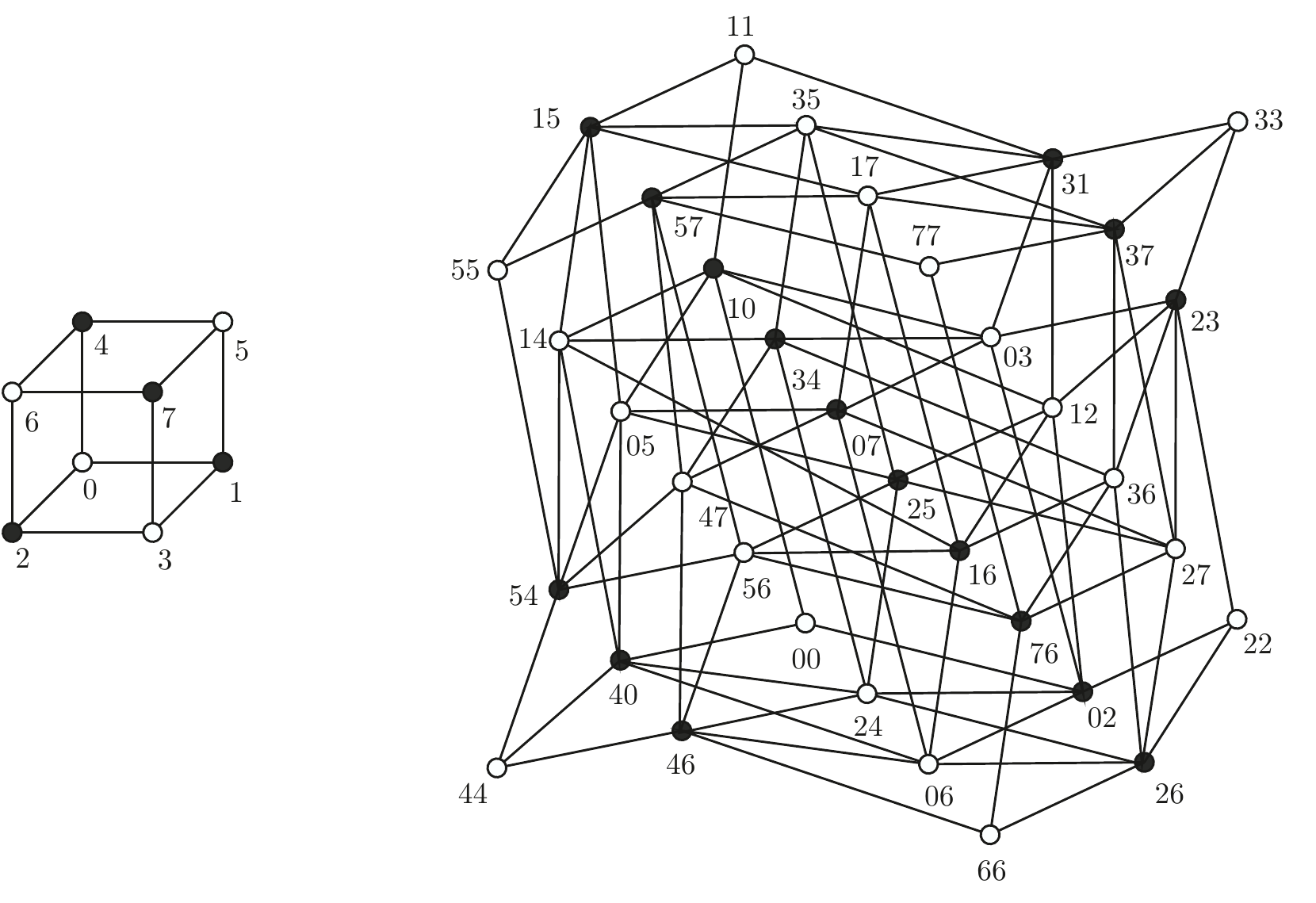}
	\end{center}
	\vskip-.5cm
	\caption{Left: The hypercube $Q_3$. Right: The 2-supertoken graph ${\cal F}_2(Q_3)$. The white vertices form a maximum independent set.}
	\label{fig:FF(Q3)}
\end{figure}

\begin{theorem}
\label{th:indFFk(G)}
Let $G$ be a graph with $n$ vertices and chromatic number $\chi$. For $i=1,\ldots,\chi$, let $c_i$ be
the number of vertices in the color class ${\cal C}_i$  of a $\chi$-coloring. 
Let ${\cal P}_1,\ldots,{\cal P}_{\sigma}$, with $\sigma\in [\lceil\chi/k\rceil,\chi]$, be a partition of the color classes. If ${\cal P}_j=\{{\cal C}_{i_1},\ldots,{\cal C}_{i_{\tau_j}}\}$, where $\tau_j\in[1,k]$, let  
\begin{equation}
\label{|P|}
[\#{\cal P}_j]:=
\prod_{\stackrel{{h\in[1,\tau_j],}\,\pi_h>0} 
{\pi_1+\cdots+\pi_{\tau_j}=k,
\pi_1\le \cdots \le\pi_{\tau_j}}}{c_{i_{h}}+\pi_h-1\choose \pi_h},
\end{equation}
where $\pi_h$ is the number of tokens in each ${\cal C}_{i_h}$ for $h=1,\ldots,\tau_j$.
Thus, the independence number of the $k$-supertoken graph  satisfies the bound
\begin{equation}
\label{alphaFF}
   \alpha({\cal F}_k(G))\ge \sum_{j=1}^{\sigma}[\#{\cal P}_j].
\end{equation} 
\end{theorem}

\begin{proof}
Let us exhibit an independent set of ${\cal F}_k(G)$ with such many vertices.
Recall that ${c_{i_{h}}+\pi_h-1\choose \pi_h}$ is the number of $c_{i_h}$ vertices taken $\pi_h$ at a time. 
Given some ${\cal P}_j=\{{\cal C}_{i_1},\ldots,{\cal C}_{i_{\tau_j}}\}$, with $\tau_j\in[1,k]$, consider the positive integers $\pi_1,\ldots,\pi_{\tau_j}$ such that $\pi_1+\cdots+\pi_{\tau_j}=k$ and $\pi_1\le\cdots\le\pi_{\tau_j}$.  Let $\vu$ and $\vv$ be two vertices of ${\cal F}_k(G)$ representing two different configurations of $k$ tokens in some vertices of ${\cal P}_j=\{{\cal C}_{i_1},\ldots,{\cal C}_{i_{\tau_j}}\}$, where we take  $\pi_h$ tokens in each ${\cal C}_{i_h}$ for $h=1,\ldots,\tau_j$. Then, such configurations differ at least in two tokens placed in different vertices of the same color class ${\cal C}_{i_h}$. Since such vertices are independent, a path between $\vu$ and $\vv$ has a length of at least two, and hence, $\vu$ and $\vv$ are non-adjacent.\\
Now suppose that $\vu$ and $\vv$ are two vertices of 
${\cal F}_k(G)$ corresponding to configurations of $k$ tokens in different sets ${\cal P}_j$ and ${\cal P}_k$ of the partition.
Then, to go from $\vu$ to $\vv$, 
we need to move at least two tokens between color classes in  ${\cal P}_j$ and ${\cal P}_k$. Thus, $\vu$ and $\vv$ are again non-adjacent.\\
We obtain the expected result by summing over all configurations in this way.
\end{proof}

For the sake of clarity, consider the two extreme cases $\tau_j=1$ and $\tau_j=k$. 
\begin{itemize}
\item 
If $\tau_j=1$, for every $j=1,\ldots,\sigma$, we have that $\sigma=\chi$ and   ${\cal P}_j={\{{\cal C}_{i_j}\}}$,
with {${\cal C}_{i_j}$} having {$c_{i_j}$} vertices. Then, $[\#{\cal P}_j]={{c_{i_j}}+k-1\choose k}$ (we take the $k$ tokens in the same color class).
\item 
If $\tau_j=k$ for some $j=1,\ldots,\sigma$, we have  ${\cal P}_j=\{{\cal C}_{i_1},\ldots,{\cal C}_{i_{k}}\}$ and \eqref{|P|} gives $[\#{\cal P}_j]=\prod_{h=1}^k c_{i_h}$ (we take one token for each color class). 
\end{itemize}

Let us see a pair of examples.
\begin{example}
In $G=Q_3$, the cube graph, we have $n=8$, $\chi(G)=2$, and $c_1=c_2=4$.
Then, its 2-supertoken ${\cal F}_2(G)$ has ${9\choose 2}=36$ vertices, and, depending on the two possible partitions (since any number associated with the partition class must be at most $k=2$), we have the following table. 
\begin{center}
\renewcommand{\arraystretch}{1.3}
\begin{tabular}{|c|c|}
\hline
 Color class    &  Bound \eqref{alphaFF}\\[-.2cm]
 partition    &  \\
 \hline
 $1+1$   &  $2\cdot{4+2-1\choose 2}=20$\\
 \hline
 $2$      &   $4^2=16$\\
 \hline
\end{tabular} 
    \end{center}
Note that the color class partition $1+1$ means that we have two color classes (for example, ${\cal C}_1$ and  ${\cal C}_2$), while color class $2$ means that we only have one color class (for example, ${\cal C}_1$).
Thus, $\alpha({\cal F}_2(G))\ge 2{5\choose 2}=20$. In fact, since ${\cal F}_2(G)$ is bipartite, equality holds. In Figure \ref{fig:FF(Q3)}, we show $Q_3$ and ${\cal F}_2(Q_3)$. 
In the area of information theory, if $Q_3$ is used as our confusability graph, the information rate is  $\log_2\alpha(Q_3)=2$. Moreover, in ${\cal F}_2(G)$, the information rate of strings of length $k=2$ per symbol is
at least
$$
\frac{\log_2(\alpha({\cal F}_2(G)))}{2}=\log_2\sqrt{20}\approx 2.1609.
$$
In contrast, it can be shown that
$$
\frac{\log_2(\alpha({\cal F}_k(G)))}{k}=\log_2\sqrt[k]{\log_2 2{k+3\choose k}}<2,
$$
that is, in this case, the minimum information rate is attained with multisets of $k=2$ symbols.
\end{example}

\begin{figure}[t]
	\begin{center}
		\includegraphics[width=6cm]{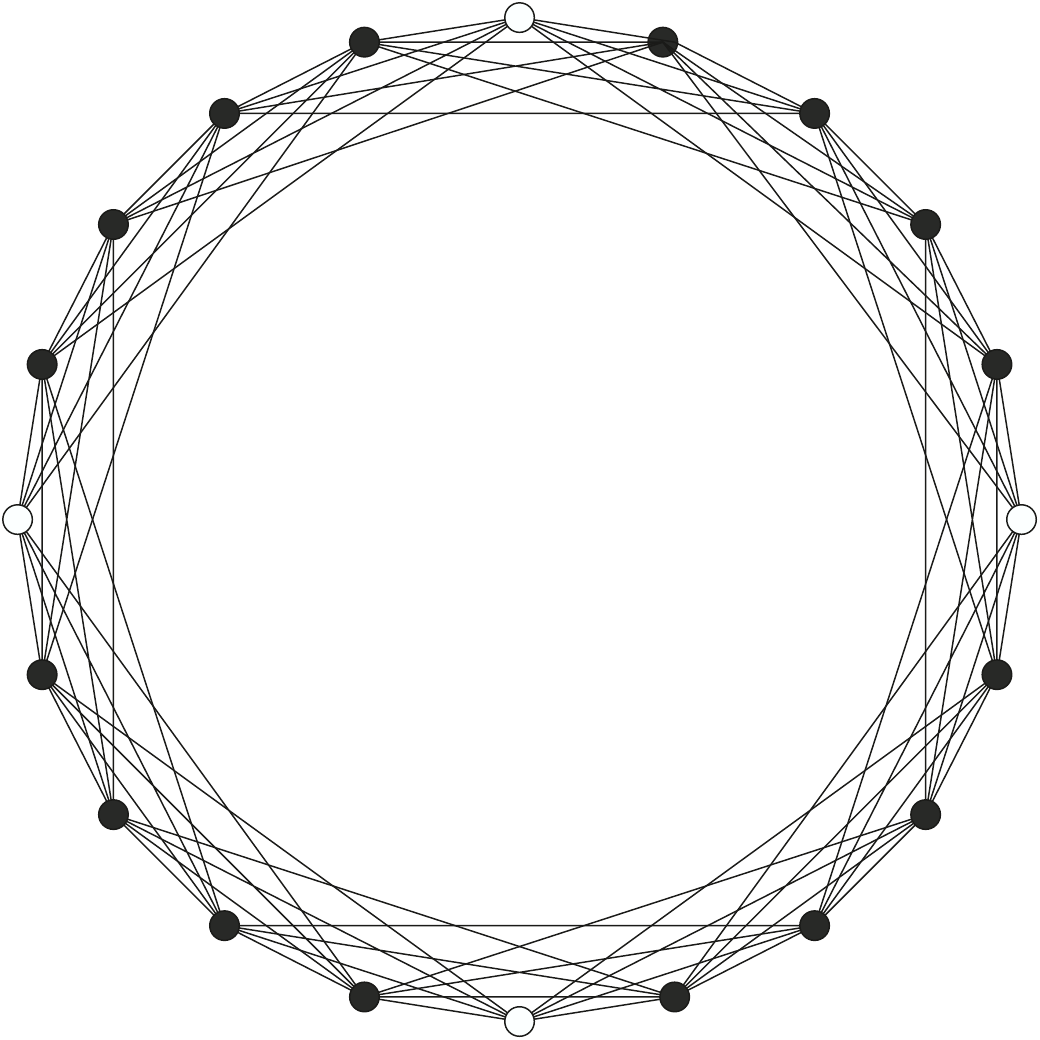}
	\end{center}
	\caption{The graph $C_{20}^4$ with $\alpha(C_{20}^4)=4$ (the white vertices). We have $\alpha({\cal F}_3(C_{20}^4))\geq 104$.}
	\label{fig:C_20^4}
\end{figure}

\begin{example}
In $G=C_{20}^4$, the 4-th power graph of the cycle on 20 vertices (see Figure \ref{fig:C_20^4}), we have $n=20$, $\chi(G)=5$, and $c_1=\cdots=c_5=4$.
Then, its 3-supertoken ${\cal F}_3(G)$ has ${20+3-1\choose 3}=1540$ vertices. Then, depending on the partition of the color classes, we obtain the following bounds for $\alpha({\cal F}_3(G))$: 

\begin{center}
\renewcommand{\arraystretch}{1.3}
\begin{tabular}{|c|c|}
\hline
Color class    &  Bound \eqref{alphaFF}\\[-.2cm]
 partition    &  \\
 \hline
 $1+1+1+1+1$\   ({$\sigma=$5})  &  $5\cdot{4+3-1\choose 3}=100$\\
 \hline
 $2+1+1+1$ \   ({$\sigma=$4})    &   $4\cdot{4+2-1\choose 2}+3\cdot{4+3-1\choose 3}=100$\\
 \hline
 $2+2+1$ \  ({$\sigma=$3})       &  $2\cdot4\cdot{4+2-1\choose 2}+{4+3-1\choose 3}=100$ \\
 \hline
 $3+1+1$ \  ({$\sigma=$3})       &  $4^3+2\cdot{4+3-1\choose 3}=104$  \\
 \hline
 $3+2$ \   ({$\sigma=$2})        & $4^3+4\cdot{4+2-1\choose 2}=104$\\
 \hline
\end{tabular}
    \end{center}
    
 Then, $\alpha({\cal F}_3(G))\ge 104$. 
Moreover, if $G$ is used as our confusability graph, the information rate is  $\log_2\alpha(G)=2$, whereas, in ${\cal F}_3(G)$, the information rate of strings of length $k=3$ per symbol is at least
$$
\frac{\log_2(\alpha({\cal F}_3(G)))}{3}\ge \log_2\sqrt[3]{104}\approx 2.2334.
$$
\end{example}

These examples and some empirical evidence suggest that the maximum value in \eqref{alphaFF} is obtained when the values $\pi_1,\ldots,\pi_{\tau_j}$ are as equal as possible.

\begin{theorem}
\label{th:bip}
Let $G$ be a bipartite graph with stable sets $C_1$ and $C_2$, with cardinalities $|C_i|=c_i$ for $i=1,2$ such that $c_1\le c_2$. Then, given an integer $k\ge 2$,
the $k$-supertoken ${\cal F}_k(G)$ is bipartite with independence number satisfying
\begin{equation}
\alpha({\cal F}_k(G)) \ge \sum_{i=0}^{\lfloor k/2\rfloor}{c_1+2i-1\choose 2i}{c_2+k-2i-1\choose k-2i}.
\label{alpha(FFk(bip0)} 
\end{equation}
\end{theorem}

\begin{proof}
The proof follows the same lines of reasoning as in Fabila-Monroy, Flores-Peñaloza, Huemer, Hurtado, Urrutia, and Wood \cite{ffhhuw12} (for bipartiteness), and de Alba, Carballosa, Leaños, and Rivera \cite{aclr20} (for the independence number), but now using multisets instead of sets.
For instance, to prove that ${\cal F}_k(G)$ is bipartite, let $A,B$ be multisets with elements in $C_1\cup C_2$ (representing vertices of ${\cal F}_k(G)$). Then, the stable sets of the $k$-supertoken graph ${\cal F}_k(G)$ are 
\begin{align*}
{\cal S}_1 &=\{A : |A|=k,  \mbox{ $A$ has an odd number of elements of $C_1$}\},\\
{\cal S}_2 &=\{B : |B|=k,  \mbox{ $B$ has an even number of elements of $C_1$}\}.
\end{align*}
Indeed, since $G$ is bipartite, the vertices adjacent to a vertex $A\in {\cal S}_1$ are obtained by `moving' a token from $C_1$ to $C_2$ {or from $C_2$ to $C_1$. In either case,} this changes the parity of the number of tokens in $C_1$ and, hence, such vertices must belong to ${\cal S}_2$. Analogously, all the vertices adjacent to vertex $B\in {\cal S}_2$ must be in
${\cal S}_1$. 
Finally, notice that the right side of \eqref{alpha(FFk(bip0)}
is just the cardinality of 
${\cal S}_2$.
\end{proof}

Notice that, when $k$ is even, the formula in  \eqref{alpha(FFk(bip0)} is symmetric on the variables $c_1$ and $c_2$. {Note also that, since $G$ is always a subgraph of ${\cal F}_k(G)$, $G$ is bipartite if and only if ${\cal F}_k(G)$ is.}

\begin{lemma}
\label{lem:Hall}
Let $G$ be a bipartite graph with independent sets $C_1$, $C_2$, whose cardinalities satisfy $c_1\le c_2$, as in Theorem \ref{th:bip}. If all the vertices of $C_1$ have a degree at least $\delta$ whereas all the vertices of $C_1$ have a degree at most $\delta$, then the independence number of $G$ is $\alpha(G)=c_2$.
\end{lemma}

\begin{proof}
 As shown in de Alba, Carballosa, Leaños, and Rivera \cite[Lem. 3.2]{aclr20}, the result holds when there exists a matching $M$ of $C_1$ into $C_2$. (That is, an independent set of edges containing all the vertices of $C_1$). In turn, by applying the classical Hall's Theorem \cite{h35} and using the hypothesis, such a matching $M$ exists because, if $m$ is the number of edges between $S\subset C_1$ and its neighborhood $N(S)\subset C_2$, we have 
\begin{align*}
&|N(S)| \ge \frac{m}{\delta} \ge |S|.
\qedhere \end{align*}
\end{proof}

In the commented case of $\alpha(G)=c_2=|C_2|$, computer evidence leads us to pose the following conjecture.

\begin{conjecture}
\label{conj:alpha-bip}
In the settings of Theorem \ref{th:bip}, if $\alpha(G)=|C_2|$, then the independence number of ${\cal F}_k(G)$ attains the upper bound in \eqref{alpha(FFk(bip0)}.
\end{conjecture}

For even cycles $C_{2c}$, we have $c_1=c_2=c$ and \eqref{alpha(FFk(bip0)} gives the values in Table \ref{tab:ind-num} depending on $k\ge 0$, {together with the type of sequence in The Online Encyclopedia of Integer Sequences \cite{Sl}}
(for the trivial value $k=0$, the formula gives 1).

\begin{table}
\begin{center}
\begin{tabular}{|c|cccccccccc|c|}
\hline
$c\backslash k$ & 0 & 1 & 2 & 3 & 4 & 5 & 6 & 7 & 8 & 9 & OEIS\\
\hline
\hline
1 & 1 & 1 & 2 & 2 & 3 & 3 & 4 & 4 & 5 & 5 & A008619\\
 \hline
2 & 1 & 2 & 6 & 10 & 19 & 28 & 44 & 60 & 85 & 110 & A005993\\
 \hline
3 & 1 & 3 & 12 & 28 & 66 & 126 & 236 & 396 & 651 & 1001 & A005995\\
 \hline
4 & 1 & 4 & 20 & 60 & 170 & 396 & 868 & 1716 & 3235 & 5720 & A018211\\
\hline
5 & 1 & 5 & 30 & 110 & 365 & 1001 & 2520 & 5720 &121905 & 24310 & A018213\\
\hline
6 & 1 & 6 & 42 & 182 & 693 & 2184 & 6216 & 15912 & 37854 & 83980 & A062136\\
\hline
7 & 1 & 7 & 56 & 280 & 1204 & 4284 & 13608 & 38760 & 101850 & 248710 & ---\\
\hline
\end{tabular}
\end{center}
\caption{Bounds for the independence number of the $k$-supertoken of the cycle on $2c$ vertices.}
\label{tab:ind-num}
\end{table}

As a consequence of Theorem \ref{th:indFFk(G)}, 
we have the following result.
\begin{corollary}
Let $G$ be a graph as in Theorem \ref{th:indFFk(G)}. Then, the Shannon capacity of its $k$-supertoken graph satisfies
\begin{equation}
\label{ShannonFF}
   \Theta({\cal F}_k(G))\ge \sum_{j=1}^{\sigma}[\#{\cal P}_j],
\end{equation} 
where the numbers $[\#{\cal P}_j]$ are defined in \eqref{|P|}.
\end{corollary}


{In the following result, we show that exact values of the independence number can be given in the case of 2-supertoken graphs of cycles.}

\begin{theorem}
\label{th:alpha(p=1)}
The independence number of the $2$-supertoken graph of the cycle $C_n$, for $n\ge 2$ (where $C_2\cong K_2$) satisfies the following equalities:
\begin{align*}
 \alpha({\cal F}_2(C_n)) &=
 \left\{ 
 \begin{array}{ll}
r(n+2) & \mbox{ if } n=4r \mbox{ or } n=4r+1,\\ 
(r+1)n & \mbox{ if } n=4r+2 \mbox{ or } n=4r+3.
 \end{array}
 \right. 
\end{align*}
\end{theorem}
\begin{proof}
    First, let us describe some independent sets for both cases
    (all arithmetic is modulo $n$).
    \begin{itemize}
        \item[$(a)$] 
        If $n=4r$ or $n=4r+1$, an independent set of $rn+2r$ vertices is
        \begin{align*}
        ii, i(i+2),\ldots, i(i+2[r-1])\quad & \mbox{for $i=0,1,\ldots,n-1$,}\\
        \mbox{and\quad } i(i+2r)\quad & \mbox{for $i=0,1,\ldots,2r-1$.}
       \end{align*}
       See the examples of ${\cal F}_2(C_8)$ and ${\cal F}_2(C_9)$ in Figure \ref{fig:FF2-C8+C9}.
       \item[$(b)$] 
        If  $n=4r+2$ or $n=4r+3$, an independent set of $(r+1)n$ vertices is
        \begin{align*}
        ii, i(i+2),\ldots, i(i+2r)\quad & \mbox{for $i=0,1,\ldots,n-1$.}
       \end{align*}
       See the examples of ${\cal F}_2(C_6)$ and ${\cal F}_2(C_7)$ in Figure \ref{fig:FF2-C6+C7}.
    \end{itemize}
    To prove that these are maximal independent sets, we distinguish the cases of $n$ even and $n$ odd.
    
    \begin{itemize}
    \item[$(i)$] 
    If $n$ is even, $n=4r$ or $n=4r+2$, the obtained supertoken graph is bipartite, and Theorem \ref{th:bip} applies. Then,  \eqref{alpha(FFk(bip0)} for $k=2$ and $c_1=c_2=c$ gives $\alpha({\cal F}_2(C_{2c}))\ge r(n+2)$ if $c$ is even $(n=4r)$, and 
    $\alpha({\cal F}_2(C_{2c}))\ge (r+1)n$ if $c$ is odd $(n=4r+2)$.
    We follow the notation in Theorem \ref{th:bip} to show that such bounds attain equality. So, with $C_1=\{0,2,\ldots\}$ and $C_2=\{1,3,\ldots\}$ denoting the independent sets of $C_n$, the independent set  in $(a)$ (for $n=4r$) or in $(b)$
    (for $n=4r+2$) is 
    $$
    {\cal S}_2=\{A: |A|=2, \ A \mbox{\ has an even number (0 or 2) of elements of } C_1\}.
    $$
    Then, if $n=4r$, we have $|{\cal S}_2|=rn+2r=4r^2+2r$ and
     $|{\cal S}_1|=rn=4r^2$, whereas, if $r=4r+2$, then $|{\cal S}_2|=(r+1)n=4r^2+6r+2$ and $|{\cal S}_1|=rn+2r+1=4r^2+4r+1$.
     Thus, in both cases, $|{\cal S}_1|<|{\cal S}_2|$, ${\cal S}_1$ has vertices with degree 4, whereas ${\cal S}_2$ has vertices of degree 4 and 2. Hence, Lemma \ref{lem:Hall} applies and $\alpha({\cal F}_2(C_{2c}))=|{\cal S}_2|$.
\item[$(ii)$] 
If $n$ is odd, $n=4r+1$ or $n=4r+3$, the supertoken graph ${\cal F}_2(C_n)$ can be seen as a lift on the group $\Z_n$ of the voltage graph on $\kappa=(n+1)/2$ vertices shown in Figure \ref{fig:base-graph}. The polynomial (tridiagonal) matrix $\BB(z)$, where $z=e^{r\frac{i  2\pi}{n}}$, for $r=0,\ldots, n-1$, and its similar matrix $\BB^*(r)$
    are 
    $$
\BB(z) =
\left(
	\begin{array}{cccccc}
		0 & 1+z^{-1} & 0 & 0 & \ldots & 0 \\
		1+z & 0 & 1+z^{-1} & 0 & \ldots & 0 \\
		0 & 1+z & 0 & 1+z^{-1} & \ddots & 0 \\
		0 & 0 & 1+z & \ddots & \ddots  & 0 \\
		\vdots  & \vdots & \ddots &\ddots & 0 & 1+z^{-1}\\
		0 & 0 & \ldots &0 & 1+z & z^{\kappa-1}+z^{-\kappa+1}
\end{array}
\right)\cong
$$
	\begin{equation}
		\BB^*(r)=
		{\scriptstyle{\left(
		\begin{array}{cccccc}
			0 & 2\cos(\frac{r\pi}{n}) & 0 & 0 & \ldots & 0\\
			2\cos(\frac{r\pi}{n}) & 0 & 2\cos(\frac{r\pi}{n}) & 0 & \ldots & 0\\
			0 & 2\cos(\frac{r\pi}{n}) & 0 & 2\cos(\frac{r\pi}{n}) & \ddots& 0\\
			0 & 0 & 2\cos(\frac{r\pi}{n}) & \ddots &  \ddots & 0\\
			\vdots & \vdots & \ddots & \ddots
			& 0 & 2\cos(\frac{r\pi}{n}) \\[.1cm]
			0 & 0 & \ldots & 0 & 2\cos(\frac{r\pi}{n}) & 2(-1)^r\cos(\frac{r\pi}{n})
		\end{array}
		\right).}}
		\label{B-star}
	\end{equation}
For details about lift graphs and their spectra, see Dalf\'o, Fiol, Miller, Ryan, and \v{S}ir\'a\v{n} \cite{dfmrs17}, or Reyes, Dalf\'o, Fiol, and Messegu\'e \cite{rdfm23}. Now, by using the results of Losonczi \cite{l92} (adjusted to our case) the eigenvalues of $\BB^*(r)$, and, hence, of the adjacency matrix of ${\cal F}_2(C_n)$, are
\begin{equation}
\label{eq:eigenvals}
\lambda(r,k)=4(-1)^{r+1}\cos\left(\frac{r\pi}{n}\right)\cos\left(\frac{2k\pi}{n+2}\right), 
\end{equation}
for $r=0,\ldots,n-1$, and $k=1,\ldots,\lceil n/2\rceil$.
    Then, we again distinguish two cases:
    \begin{enumerate}
        \item[$(1)$] 
        If $n=4r+1$, the supertoken graph ${\cal F}_2(C_n)$ has $N=8r^2+6r+1$, and spectrum with $(N+1)/2$ positive eigenvalues and $(N-1)/2$ negative eigenvalues. Thus, Cvetkovi\'c bound \cite{c71} gives $\alpha\le (N-1)/2=4r^2+3r=r(n+2)$, which corresponds to the size of the given independent set.
        \item[$(2)$] 
        If $n=4r+3$, the supertoken graph ${\cal F}_2(C_n)$ has $N=8r^2+14r+6$ vertices, and spectrum with $N/2$ positive eigenvalues and $N/2$ negative eigenvalues. Thus, Cvetkovi\'c bound gives $\alpha\le N/2=4r^2+7r+3=(r+1)n$, corresponding to our independent set's size.
    \end{enumerate}
    \end{itemize}
    
    To prove $(2)$, the case $(1)$ being similar, we note that, given $r$, the function $\phi(k)=\lambda(r,k)$ is monotone in the interval $[1,2r+2]$ (increasing when $r=0,\pm 2,\pm 4,\ldots$ $\mod n$ and decreasing when $r=\pm 1,\pm3,\ldots$ $\mod n$), with a zero at $k=r+1+\frac{1}{4}$. See Figure
    \ref{fig:lambda(r,k)} for the case $n=11$ $(r=2)$. Then, for each given $r$, the $r+1$ values of $\lambda(r,k)$ for $k=1,2,\ldots,r+1$ and for $k=r+2,r+3,\ldots,2r+2$ have opposite signs. Then, for each (positive or negative) sign, we have a total of $(r+1)n=4r^2+2r+3=N/2$ eigenvalues, as claimed.
\end{proof}

\begin{figure}[t]
	\begin{center}
		\includegraphics[width=10cm]{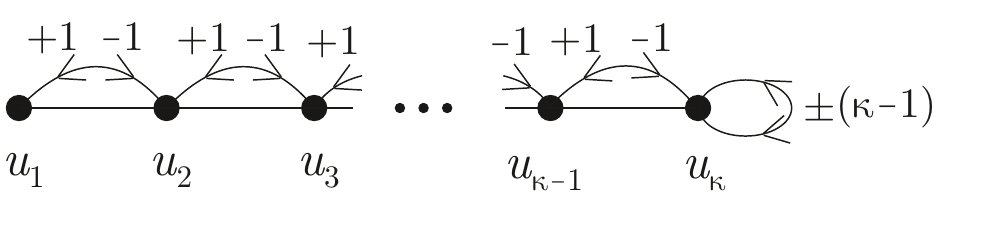}
	\end{center}
	\vskip-1cm
	\caption{The base graph of the graphs ${\cal F}_2(C_n)$ for $n$ odd.}
	\label{fig:base-graph}
\end{figure}

\begin{figure}[t]
	\begin{center}
		\includegraphics[width=10cm]{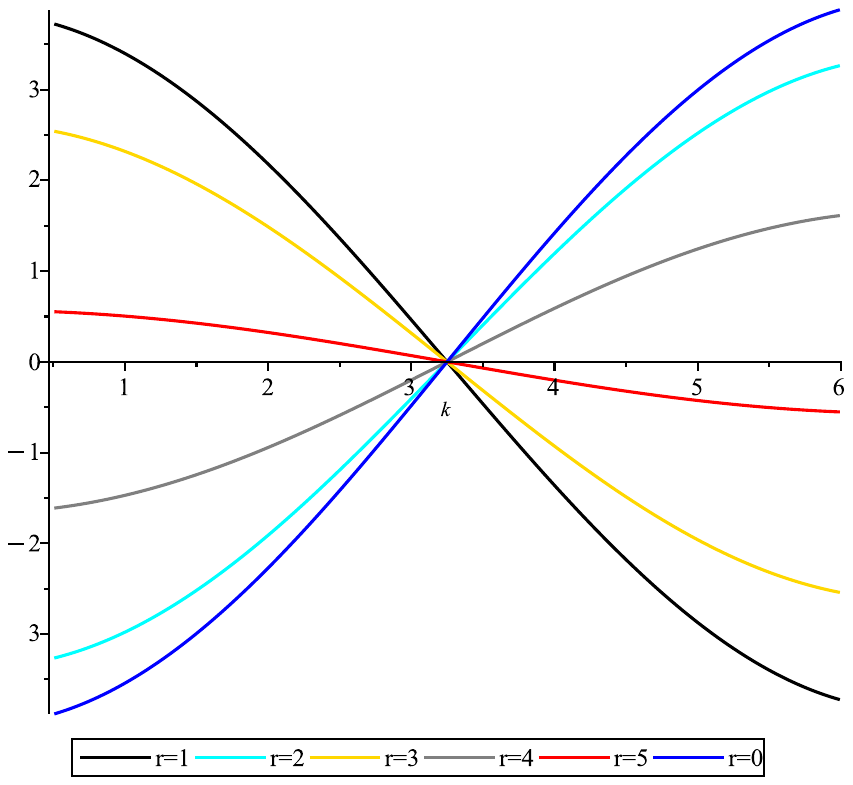}
	\end{center}
	\vskip-6cm
	\caption{The function $\lambda(r,k)$  giving the eigenvalues of ${\cal F}_2(C_{11})$.}
	\label{fig:lambda(r,k)}
\end{figure}

\begin{figure}[t]
	\begin{center}
	\includegraphics[width=11cm]{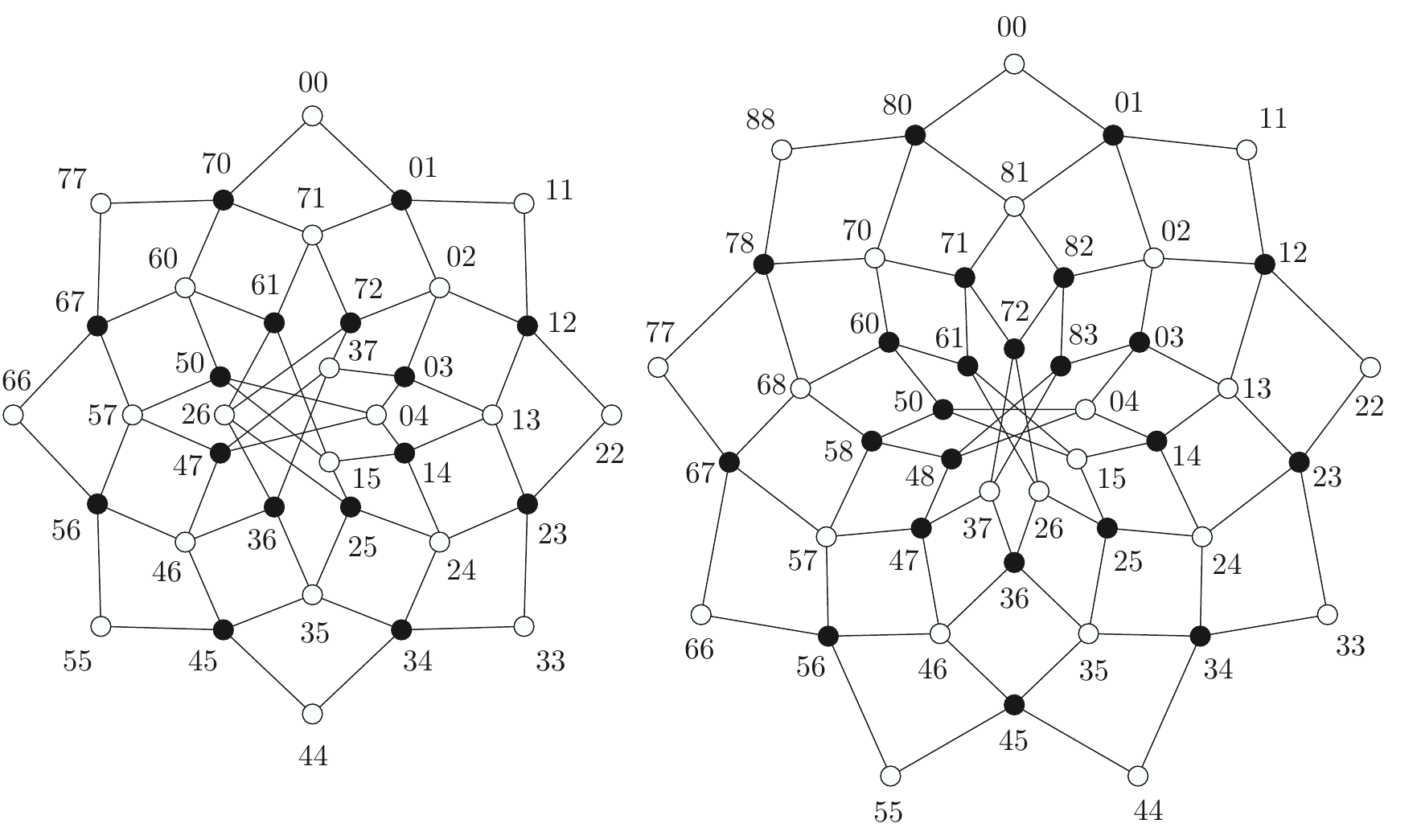}
	\end{center}
	\vskip-.5cm
	\caption{The graphs ${\cal F}_2(C_8)$ and ${\cal F}_2(C_9)$. The white vertices form independent sets.}
	\label{fig:FF2-C8+C9}
\end{figure}

\begin{figure}[t]
	\begin{center}
	\includegraphics[width=11cm]{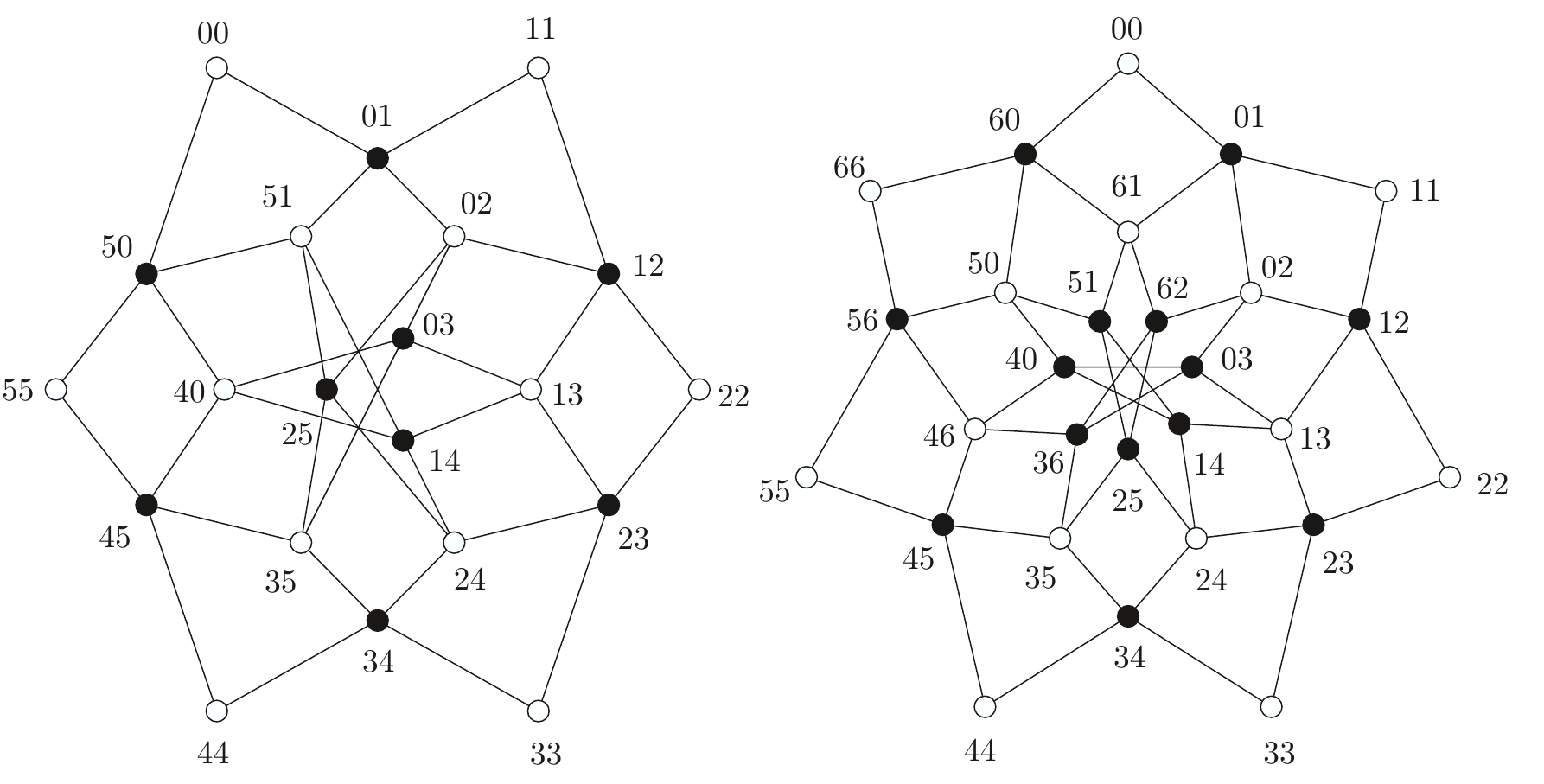}
	\end{center}
	\vskip-.5cm
	\caption{The graphs ${\cal F}_2(C_6)$ and ${\cal F}_2(C_7)$. The white vertices form independent sets.}
	\label{fig:FF2-C6+C7}
\end{figure}

\begin{example}
When $n=7$, the $4\times 4$ polynomial matrix corresponding to the base graph of Figure \ref{fig:base-graph} with $\kappa=4$ is
	\begin{equation}
		\BB(z)=\left(
		\begin{array}{cccc}
			0 & 1+\frac{1}{z} & 0 & 0 \\[.1cm]
			1+z & 0 & 1+\frac{1}{z} & 0 \\[.1cm]
			0 & 1+z & 0 & 1+\frac{1}{z}\\[.1cm]
			0 & 0 & 1+z & z^3+\frac{1}{z^3}
		\end{array}
		\right),
	\end{equation}
	giving, for $z=e^{i\frac{r2\pi}{7}}$ and $r=0,1,\ldots,6$,
	the eigenvalues of ${\cal F}_2(C_7)$ shown in Table \ref{table7}.
  \end{example}
  
  \begin{table}[t]
		\begin{center}
			\begin{tabular}{|c|cccc|}
				\hline
				$\omega=e^{i\frac{2\pi}{7}}$, $z=\omega^r$ & $\lambda_{r,1}$  & $\lambda_{r,2}$ & $\lambda_{r,3}$ & $\lambda_{r,4}$\\
				\hline\hline
				$\spec(\BB(\omega^0))$ & 3.759 & 2 &  $-3.064$ & $-0.6948$ \\
				\hline
				$\spec(\BB(\omega^1))=\spec(\BB(\omega^6))$ & 2.761 & 0.6258  & $-1.802$ & $-3.387$ \\
				\hline
				$\spec(\BB(\omega^2))=\spec(\BB(\omega^5))$ &  2.344 & 1.247  & $-0.4331$ & $-1.910$\\
				\hline
				$\spec(\BB(\omega^3))=\spec(\BB(\omega^4))$ &  0.6818 & 0.1546  & $-0.4450$ & $-0.8364$\\
				\hline
			\end{tabular}
         \caption{All the eigenvalues of the $2$-supertoken ${\cal F}_2(C_7)$.}  \label{table7}
		\end{center}
    
\end{table}

  \begin{example}
  For the case of $n=9$, the eigenvalues of ${\cal F}_2(C_9)$ given by $\eqref{eq:eigenvals}$ are shown in  Table \ref{tab:FF2(C9)}.
  \end{example}
  
  \begin{table}[t]
  \begin{center}
			\begin{tabular}{|c|rrrrr|}
				\hline
				$r \backslash k$ & $1$  & $2$ & $3$ & $4$ & $5$\\
				\hline\hline
				0 & $-3.650$ & $-1.662$ & $0.5693$ &  $2.619$ & $3.838$  \\
				\hline
				$1,8$ & $3.162$ & $1.561$  & $-0.5349$ & $-2.461$ & $-3.606$ \\
				\hline
				$2,7$ &  $-2.577$ & $-1.273$  & $0.4361$ & $2.007$ & $2.940$\\
				\hline
				$3,6$ &  $1.682$ & $0.8308$  & $-0.2846$ & $-1.310$ & $-1.919$\\
    \hline
			$4,5$	& $-0.5843$ &  $-0.2885$ & $0.0988$  & $0.4548$ & $0.6664$ \\
				\hline
			\end{tabular}
            \caption{All the eigenvalues of the $2$-supertoken ${\cal F}_2(C_9)$ given by \eqref{eq:eigenvals}.}
            \label{tab:FF2(C9)}
		\end{center}
\end{table}

 \section{{Clique number of supertoken graphs}}
\label{sec:clique}
 
 The concept of cliques is a fundamental topic in graph theory. In a graph $G$, a clique is a subset of vertices where every pair is mutually adjacent. The clique number $\omega(G)$ represents the size of the largest clique in $G$. \\
 In \cite{ffhhuw12}, Fabila-Monroy, Flores-Pe\~{n}aloza, Huemer, Hurtado, Urrutia, and Wood proved the following result for token graphs.
 \begin{theorem} \cite{ffhhuw12}  
 $\omega (F_k(G))= \min\{\omega(G), \max\{n -k + 1, k + 1\}\}$.   
 \end{theorem}
 {In contrast, for the case of supertoken graphs, the situation is different. See an example in Figure \ref{fig:F_2^2(K4)}}. For supertoken graphs, we establish the following result.
\begin{proposition}\label{prop:types}
   Let $G$ be a graph, and ${\cal F}_k(G)$ be the $k$-supertoken graph of $G$. If $X=\{A,B,C\}$  is a clique in ${\cal F}_k(G)$, then there exists a clique of size 3 in G. Furthermore, 
 $X$ can only be one of the following two types:  
 \begin{itemize}
   \item [Type 1:] $k-1$ elements of $A$, $B$, and $C$ remain fixed, say $\{a_1, \dots, a_{k-1}\}$, while the $k$-th element, denoted as $a$, $b$, and $c$ for $A$, $B$, and $C$ respectively, corresponds to the vertices of a $3$-clique in $G$. That is, $A = \{a_1, a_2, \dots, a_{k-1}, a\}$, $B = \{a_1, a_2, \dots, a_{k-1}, b\}$, and $C = \{a_1, a_2, \dots, a_{k-1}, c\}$, where the clique in $G$ is $\{a, b, c\}$.  
\item [Type 2:] $k-2$ elements of $A$, $B$, and $C$ remain fixed, say $\{a_1, \dots, a_{k-2}\}$, while the remaining two positions in $A$, $B$, and $C$ correspond to the vertices of a $3$-clique in $G$. Specifically, $A = \{a_1, a_2, \dots, a_{k-2}, a_{k-1}, a\}$, $B = \{a_1, a_2, \dots, a_{k-2}, a_{k-1}, b\}$, and $C = \{a_1, a_2, \dots, a_{k-2}, b, a\}$, where the clique in $G$ is $\{a, a_{k-1}, b\}$.  
 \end{itemize}
 \end{proposition}
 \begin{proof}
Let $ X = \{A, B, C\} $ be a 3-clique in $ \mathcal{F}_k(G) $. By the definition of adjacency in $ \mathcal{F}_k(G) $, the vertices $ A $ and $ B $ differ in one element. This means $ A = \{a_1, a_2, \dots, a_{k-1}, a\} $ and $ B = \{a_1, a_2, \dots, a_{k-1}, b\} $, where $ a, b \in V(G) $ and $ (a, b) \in E(G) $.  
Since $ C $ is adjacent to $ A $, there are two possible cases to consider: 
\begin{itemize}
\item[Case 1:] The token placed at vertex $ a $ moves to vertex $ c $, meaning $ C = \{a_1, a_2, \dots, a_{k-1}, c\} $.  
As $ C $ is also adjacent to $ B $, the token in {$c$} must move from $ c $ to $ b $. Consequently, the vertices $ a, b, c \in V(G) $ must form a 3-clique in $ G $. This corresponds to a Type 1 clique in $ \mathcal{F}_k(G) $, where $ k-1 $ elements remain fixed, and the $ k $-th elements $ a, b, c $ are moving. 
\item[Case 2:] The token placed at a different vertex, say $ a_{k-1} $ without loss of generality, moves to vertex $ c $, which implies $ C = \{a_1, a_2, \dots, c, a\} $.  
Since $ C $ is also adjacent to $ B $, the only way this is possible is if $ c = b $ and $ a $ is adjacent to $ a_{k-1} $ in $ G $. Thus, the vertices $ a, a_{k-1}, b \in V(G) $ must form a 3-clique in $ G $. This configuration corresponds to a Type 2 clique. This completes the proof.
\end{itemize}
\end{proof}
{Notice that the 3-clique of Type 1 is characterized by $B\cap C:=\{a_1,a_2,\ldots,a_{k-1}\}\subset A$, whereas in the 3-clique of Type 2 we have $A\subset B\cup C:=\{a_1,a_2,\ldots,a_{k-1},a,b\}$. These are precisely the two classes of 3-cliques considered in  \cite{ffhhuw12} for the token graph  
 $F_k(G)$}.
\begin{proposition}
\label{propo:cliqueK}
    Let $G$ be a graph and $X$ a clique of the $k$-supertoken graph $\mathcal{F}_k(G)$. Then, there exists a clique $K$ of $G $ such that {$|K|=|X|$}.  
\end{proposition}
\begin{proof}
The proof goes along the same lines of reasoning in 
Fabila-Monroy,  Flores-Pe\~{n}aloza, Huemer, Hurtado, Urrutia, and  Wood \cite[Th. 4]{ffhhuw12}, but working with multisets instead of sets. More precisely, the clique $X$ is such that there exists a multiset $S$ of elements in  $V(G)$ and a set $K\subset V(G)$ such that, 
either
\begin{itemize}
    \item[$(i)$] $X=\{S\uplus \{v\}: v\in K\}$ and $|S|=k-1$,
    \item[$(ii)$] $X=\{(S\uplus K)\setminus  \{v\}: v\in K\}$ and $|S|+|K|=k+1$.
    \end{itemize}
    (The `$\uplus$' symbol combines multisets, counting repeated occurrences of elements from both multisets.)
    In the first case, all the triangles induced by three vertices of $X$ are of Type 1, as described in Proposition~\ref{prop:types}, whereas, in the second case, all the triangles are of Type 2. More precisely, it can be proved that, if $X=\{A_1,A_2,\ldots,A_q\}$ is a clique in ${\cal F}_k(G)$, then, either
    \begin{itemize}
    \item[$(i)$]
       $A_1\cap A_2\subset A_i$ for every $i=1,2,\ldots,q$ (Type 1),
       \item[$(ii)$]
       $A_i\subset A_1\cup A_2$ for every $i=1,2,\ldots,q$ (Type 2).
    \end{itemize}
\end{proof}

 \begin{theorem}
 \label{th:omegas}
     $\omega ({\cal F}_k(G))= \omega(G)$.
 \end{theorem}    
 \begin{proof}
 For the lower bound, we observe that $G$, which is isomorphic to $\cal{F}_1(G)$, is contained in any $k$-supertoken graph of $G$. Therefore, we have $\omega (\cal{F}_k (G))\geq\omega (G)$. {Moreover, the upper bound  $\omega (\cal{F}_k (G))\leq\omega (G)$
 follows from Proposition \ref{propo:cliqueK}.}
\end{proof}

\begin{figure}[t]
	\begin{center}
	\includegraphics[width=8cm]{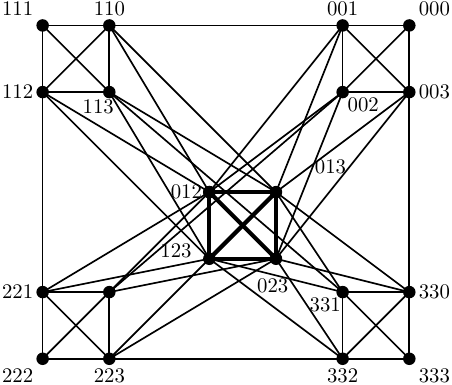}
	\end{center}
	\caption{The supertoken graph $\cal{F}_3^3(K_4)$ 
    of Figure \ref{k4_token_supertoken} with clique number $\omega(\cal{F}_3^3(K_4))=\omega(K_4)=4$. {In a thick line, there is the 4-clique of Type 2}, {whereas the other four 4-cliques are of Type 1.}}
	\label{fig:F_2^2(K4)}
\end{figure}

The clique number is closely related to several coloring parameters of a graph. 
In particular, the clique number $\omega(G)$ provides a natural lower bound for 
the chromatic number $\chi(G)$, since the vertices of any clique must receive 
distinct colors in every proper coloring. Furthermore, the chromatic number is 
bounded above by the achromatic number $\psi(G)$, which yields the well-known 
inequalities
$$\omega(G) \le \chi(G) \le \psi(G).$$

\section{{Chromatic number of supertoken graphs}}
\label{sec:chrom}

In the case of token graphs, it was shown in {Fabila-Monroy, Flores-Pe\~naloza, Huemer, Hurtado, Urrutia, and  Wood \cite{ffhhuw12}(Theorems 6 \& 7)} that the chromatic number of token graphs satisfies the bound{s}
\[
\left(\frac{1}{2}+\frac{2}{n}\right)\chi(G)-1 \leq \chi(F_k(G)) \leq \chi(G)
\]
for all $k \ge 1$.

{For supertoken graphs, we have a different situation, as shown in the following results.}
\begin{proposition}
\label{upper-bound}
    If a graph $G$ has a (vertex-)coloring with chromatic {number}  $\chi$, then the supertoken graph ${\cal F}_k(G)$ has chromatic {number}  $\chi'\le \chi$.
\end{proposition}
\begin{proof}
    Let $c:V(G)\rightarrow \{0,1,\ldots,\chi-1\}$ be a coloring of $G$. Then, to each vertex $A$ of ${\cal F}_k(G)$, we assign the color
    $$
    c'(A)=\left(\sum_{v\in A} c(v)\right) {(\text{mod } \chi)}.
    $$
   If $A$ and $B$ are adjacent vertices in ${\cal F}_k(G)$, then there exists adjacent vertices $u,v\in V$ such that, with self-explanatory notation, $A=S+u$ and $B=S+v$, where $|S|=k-1$. So, since $c(u)\not= c(v) \ {(\text{mod } \chi)}$, we have
    $$
    c'(A)=\left(\left(\sum_{x\in S} c(x)\right)+c(u)\right) {(\text{mod } \chi)} \not= \left(\left(\sum_{x\in S} c(x)\right)+c(v)\right) {(\text{mod } \chi)} =c'(B).
    $$
    and $c'$ is a $\kappa'$-coloring
    of ${\cal F}_k(G)$ for some $\kappa'\le \chi$. 
\end{proof}
In particular, if  $\chi(G)=\omega(G)$, then we have $\chi({\cal F}_k (G)) \ge \omega({\cal F}_k (G)) = \omega(G) = \chi(G)$ by Theorem \ref{th:omegas}. {So, in this case, $\chi(\purple{{\cal F}}_k(G)) =\chi(G)$.}

\begin{theorem}
    Given a graph $G=(V,E)$ with chromatic number $\chi(G)$, the supertoken graph ${\cal F}_k(G)$ has chromatic number $\chi(\purple{{\cal F}}_k(G)) =\chi(G)$.
\end{theorem}
\begin{proof}
    {From Proposition \ref{upper-bound}, we only need to prove that $\chi({\cal F}_k(G))\ge \chi (G)$.
    But this is a consequence that ${\cal F}_k(G)$ contains an induced subgraph of $G$. Then, any $\chi'$-coloring of ${\cal F}_k(G)$ induces a $\kappa$-coloring of $G$ for some $\kappa\le \chi'$. }
\end{proof}

\section{The $p$-augmented 2-token graphs of cycles}
\label{sec:aug-token-cicles}

Given $n\ge 3$ and $p\ge 0$, the \textit{$p$-augmented 2-token graph} of a cycle $C_n$, denoted by  $F_2^p(C_n)$, is constructed as follows:
We start from the 2-token graph $F_2(C_n)=F_2^0(C_n)$, where the vertices $\{i,i+1\}$ are denoted by $\{i,i+1\}^0$ (all arithmetic is modulo $n$). Then, assuming that $p$ is even (the odd case is similar), we add the vertices 
$\{i,i\}^r$ for $r=1,3,\ldots,p-1$, and $\{i,i+1\}^s$ for $s=2,4,\ldots,p$, together with the following edges:
\begin{align*}
\{i,i\}^r &\sim \{i,i+1\}^{r-1}, \{i,i-1\}^{r-1},\\
\{i,i+1\}^s & \sim \{i,i\}^{s-1}, \{i+1,i+1\}^{s-1}.
\end{align*}
{In Figure \ref{fig:F24(C5)} (left), we show the {4-}augmented 2-token $F_2^4(C_5)$ and, in Figure \ref{fig:F24(C5)} (right), the {4-}augmented 2-token $F_2^4(C_4)$, both with vertex labels $ij^k=\{i,j\}^k$.}

Some simple properties of the augmented 2-token graphs of $C_n$ are the following:
\begin{itemize}
\item 
$F_2^p(C_n)$ has $N={n\choose 2}+pn$ vertices and {$M=m{{n-2}\choose {k-1}}+2pn$} edges. 
\item 
$F_2^0(C_n)\cong F_2(C_n)$ and $F_2^1(C_n)\cong {\cal F}_2(C_n)$.
\item 
{If $n$ is even, $F_2^p(C_n)$ is a bipartite graph.}
\item 
If $p\le q$, then $F_2^{p}(C_n)$ is an induced subgraph of $F_2^{q}(C_n)$.
\item 
If $p\le q$, then the eigenvalues of $F_2^{p}(C_n)$ interlace the eigenvalues of $F_2^{q}(C_n)$.
\end{itemize}

{The construction of $F_2^p(C_n)$ sits at an interesting intersection between (standard) token graphs, configuration-space graphs, and layered state-extension models. Because we are augmenting the standard 2-token graph of a cycle with parity-stratified ``collision layers'' $(\{i,i\}^r)$ and replicated adjacency layers $(\{i,i+1\}^s)$, the object naturally encodes multi-state interactions with memory depth $(p)$. This interpretation opens several applications. So, we can interpret $F_2^p(C_n)$
as a stratified discrete configuration complex.
For more information, see Ghrist \cite{g07}.}

\begin{theorem}
The independence number of the $p$-augmented $2$-token graph of the cycle $C_n$, for $n\ge 2$ (where $C_2\cong K_2$) and different values of $p$ is shown in Table \ref{tab:alpha-augmented}.
\end{theorem}

\begin{table}[!ht]
\begin{tabular}{|c|c|c|c|c|}
\hline
$n$ 
& even $p\ge 0$ & odd $p\ge 1$ & $p=0$ & $p=1$\\
\hline \hline
$4r$ & $\left(r+\frac{p}{2}\right)n$ & $\left(r+\frac{p}{2}\right)n$ & $4r^2=\displaystyle\left\lfloor \frac{n\lfloor \frac{n}{2}\rfloor}{2}\right\rfloor$ &  $4r^2+2r=r(n+2)$ \\
$4r+1$ & $\left(r+\frac{p}{2}\right)n$ & $\left(r+\frac{p}{2}\right)n-\frac{1}{2}$ & {$4r^2+r=\displaystyle\left\lfloor \frac{n\lfloor \frac{n}{2}\rfloor}{2}\right\rfloor$} & $4r^2+3r=r(n+2)$
\\
$4r+2$ & $\left(r+\frac{p}{2}+\frac{1}{2}\right)n$ & $\left(r+\frac{p}{2}+\frac{1}{2}\right)n$ & $(2r+1)^2=\displaystyle\left\lfloor \frac{n\lfloor \frac{n}{2}\rfloor}{2}\right\rfloor$ & $4r^2+6r+2=(r+1)n$\\
$4r+3$ & $\left(r+\frac{p}{2}+\frac{1}{2}\right)n-\frac{1}{2}$ & $\left(r+\frac{p}{2}+\frac{1}{2}\right)n$ & $4r^2+5r+1{=\displaystyle\left\lfloor \frac{n\lfloor \frac{n}{2}\rfloor}{2}\right\rfloor}$ & $4r^2+7r+3=(r+1)n$\\
\hline
\end{tabular}
\vskip.5cm
\caption{The independence number of the $p$-augmented $2$-token $F_2^p(C_n)$.}
\label{tab:alpha-augmented}
\end{table}

\begin{proof}
{Notice that the case when $p=1$ corresponds to the results in Theorem \ref{th:alpha(p=1)}. {Then, the proof for a general value of $p$ goes along the same lines of reasoning as in the proof of such a theorem.}}    
\end{proof}


\begin{proposition}
Let $C_n$ be a cycle graph {with an odd number $n$ of vertices.} Then, for any integer $p\ge 0$, the spectrum of the $p$-augmented $2$-token graph $F_2^p(C_n)$ of $C_n$ contains the eigenvalues
\begin{equation}
\lambda_i=4\cos\left(\frac{(2i-1)\pi}{n+2p}\right),\quad i=1,2,\ldots,\lfloor n/2\rfloor + p.
\label{eigenvals-augmented}
\end{equation}
In particular, its spectral radius is
$$
\rho(F_2^p(C_n))=\lambda_1=4\cos\left(\frac{\pi}{n+2p}\right).
$$
\end{proposition}
\begin{proof}
Let $n=2r+1$. Then, the $p$-augmented 2-token of $C_n$ has a regular partition shown in Figure \ref{fig:quocientF_2^p(C_{2r+1})} with quotient matrix of size $r+p$ of the form
$$
\Q=\left(
\begin{array}{cccccccc}
2      & 2      & 0      &      0 & \cdots & 0      & 0      & 0\\
2      & 0      & 2      & 0      & \cdots & 0      & 0      & 0\\
0      & 2      & 0      & 2      & \cdots & 0      & 0      & 0\\
\vdots & \vdots & \vdots & \ddots & \ddots & \ddots & \vdots & \vdots\\
0      & 0      & 0      & 0      & \cdots & 2      & 0      & 2\\
0      & 0      & 0      & 0      & \cdots & 0      & 2      & 0\\
\end{array}
\right).
$$
Then, as a tridiagonal matrix, it has the eigenvalues \eqref{eigenvals-augmented} with respective eigenvectors $\vv_i$ with entries
\begin{equation}
\vv_i(j)=\cos\left(\frac{(2i-1)(2j-1)\pi}{2(n+2p)}\right), \quad 
j=1,2,\ldots,r+ p,
\label{eigenvecs-augmented}
\end{equation} 
see the results in Yueh \cite[Th. 2]{y05}. As proved by Fonseca and Kowalenko \cite{fk20}, the first to study the eigenpairs of some tridiagonal matrices was Losonczi \cite{l92}. Yueh gave a more modern presentation of his results, but he did not cite Losonczi.
{Moreover, the eigenvector $\vv_1$ of $\lambda_1$ is positive since \eqref{eigenvecs-augmented} gives $\vv_1(j)=\cos\left(\frac{(2j-1)\pi}{2(n+2p)}\right)>0$ for 
$j=1,2,\ldots,r+p$. Thus, by the Perron-Frobenius theorem, $\lambda_1$ is the spectral radius of $F_2^p(C_n)$.}
\end{proof}

\begin{proposition}
Let $C_n$ be a cycle graph with an even number $n$ of vertices. Then, for any integer $p\ge 0$, the spectrum of the $p$-augmented $2$-token graph $F_2^p(C_n)$ of $C_n$ contains the eigenvalues that are the roots of the polynomial $\phi_r(x)$, with $r=\frac{n}{2}+p$, defined recursively
as
\begin{equation}
\phi_r(x)=x\phi_{r-1}(x)-4\phi_{r-2}(x)
\label{recur-augmented}
\end{equation}
with initial values $\phi_2(x)=x^2-8$ and $\phi_3(x)=x^3-12x$.
\end{proposition}
\begin{proof}
In this case, the quotient matrix is 
$$
\Q_r=\left(
\begin{array}{cccccccc}
0 & 4 & 0 & 0 & \cdots & 0 & 0 & 0\\
2 & 0 & 2 & 0 & \cdots & 0 & 0 & 0\\
0 & 2 & 0 & 2 & \cdots & 0 & 0 & 0\\
\vdots & \vdots & \vdots & \ddots & \ddots & \ddots & \vdots & \vdots\\
0 & 0 & 0 & 0 & \cdots & 2 & 0 & 2\\
0 & 0 & 0 & 0 & \cdots & 0 & 2 & 0\\
\end{array}
\right),
$$
with characteristic polynomials $\phi_2(x)$ and $\phi_3(x)$, as indicated above. Then, the three-term recurrence is obtained by developing the determinant $|x\I-\Q_r|$ by the elements of the last column. For instance, for $r=5$, we have
\begin{align*}
\phi_5(x) & =
\left|
\begin{array}{ccccc}
     x & -4 & 0 & 0 & 0  \\
    -2 & x & -2 & 0 & 0 \\
    0 & -2 & x & -2 & 0 \\
    0 & 0 & -2 & x & -2\\
    0 & 0 & 0 & -2 & x
\end{array}
\right| = x\phi_{4}(x)+2\left|
\begin{array}{cccc}
     x & -4 & 0 & 0 \\
    -2 & x & -2 & 0 \\
    0 & -2 & x & -2 \\
    0 & 0 & 0 &  -2\\
\end{array}
\right|\\
&=x\phi_{4}(x)-4
\left|
\begin{array}{ccc}
     x & -4 & 0 \\
    -2 & x & -2 \\
    0 & -2 & x 
\end{array}
\right|
=x\phi_{4}(x)-4\phi_3(x).
\end{align*}
\end{proof}
In particular, the spectral radius of $\Q_r$ with positive eigenvector, or maximum root of $\phi_r(x)$, gives the spectral radius of $F_2^p(C_n)$. For instance, in Table \ref{tab:spec-rad-even} there is the spectral radius  of the $p$-augmented 2-token graph  $F_2^p(C_n)$  for even $n$ and $r=\frac{n}{2}+p=2,\ldots,7$.
\begin{table}[t]
  \begin{center}
			\begin{tabular}{|c|cccccc|}
				\hline
				$r=\frac{n}{2}+p$   & $2$ & $3$ & $4$ & $5$ & $6$ & $7$\\
				\hline
				$\rho(\Q_r)=\rho(F_2^p(C_n))$ & $2.8284$ & $3.4641$ & $3.6955$ &  $3.8042$ & $3.8637$ & $3.8997$ \\
				\hline
			\end{tabular}
            \vskip .5cm
            \caption{The spectral radius  of the $p$-augmented 2-token graph  $F_2^p(C_n)$  for even $n$.}
            \label{tab:spec-rad-even}
		\end{center}
\end{table}

 \begin{figure}[t]
	\begin{center}
\includegraphics[width=15cm]{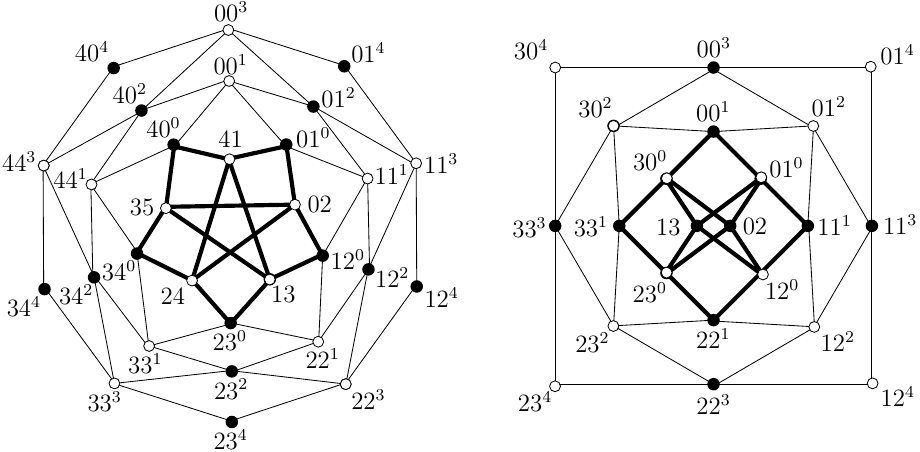}
	\end{center}
	\caption{{Left: The 4-augmented 2-token graph $F_2^4(C_5)$ of the 5-cycle with independent number $\alpha(F_2^4(C_5))=15$ (the white vertices). Right: 
    The 4-augmented 2-token graph $F_2^4(C_4)$ of the 4-cycle with independent number $\alpha(F_2^4(C_4))=12$ (the white vertices).
    In a thick line, there are the edges of $F_2(C_5)$ and $F_2(C_4)$, respectively.}}
	\label{fig:F24(C5)}
\end{figure}

 \begin{figure}[t]
	\begin{center}
\includegraphics[width=15cm]{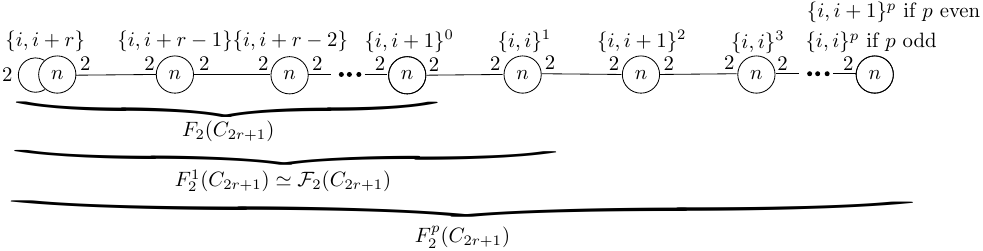}
	\end{center}
	\caption{The quotient graph of $F_2^p(C_{2r+1})$.}
	\label{fig:quocientF_2^p(C_{2r+1})}
\end{figure}
 

 


\subsection*{Acknowledgments} 
This research work has been funded by AGAUR, the Catalan Government, under project 2021SGR00434, and by MICINN, the Spanish Government, under project PID2020-115442RB-I00.
M. A. Fiol's research is also supported by a grant from the Universitat Polit\`ecnica de Catalunya with reference AGRUPS-2025. 

\subsection*{Conflict of interest}
The authors declare that they have no conflict of interest. 

\subsection*{Data availability statement} 
All the data generated or analyzed during this study are included in this article.


\end{document}